\documentclass[10pt]{amsart}

\usepackage{graphicx}
\usepackage{xcolor}

\usepackage[utf8]{inputenc}
\usepackage{amssymb,amsmath,amsthm,amsfonts}

\theoremstyle{plain}
\newtheorem{theorem}{Theorem}
\newtheorem{lemma}{Lemma}

\newtheorem*{lemma*}{Lemma}
\newtheorem*{theorem*}{Theorem}

\theoremstyle{definition}

\theoremstyle{remark} 
\newtheorem{claim}{Claim}
\newtheorem*{claim*}{Claim}

\newcommand{\N}{\mathbb{N}}
\newcommand{\R}{\mathbb{R}}

\DeclareMathOperator{\prob}{\mathbb{P}}
\DeclareMathOperator{\length}{len}
\DeclareMathOperator{\expval}{\mathbb{E}}

\newcommand{\equaldist}{\stackrel{(d)}{=}}

\newcommand{\res}{\upharpoonright}

\begin{document}

\title{Random continuum and Brownian motion}

\author{Viktor Kiss \and S{\l}awomir Solecki}

\thanks{Kiss' research supported by NSF grant DMS-1455272 and by the National Research, Development and Innovation Office---NKFIH, 
grants no.\,128273, no.\,124749, and no.\,129211}

\thanks{Solecki's research supported by NSF grants DMS-1800680 and DMS-1954069.} 


\address{Alfr{\'e}d R{\'e}nyi Institute of Mathematics\\ 
Re\'altanoda u. 13--15, H-1053 Budapest, Hungary} 

\email{kiss.viktor@renyi.mta.hu}

\address{Department of Mathematics\\
Cornell University\\
Ithaca, NY 14853, USA}

\email{ssolecki@cornell.edu}



\begin{abstract}
We describe a probabilistic model involving iterated Brownian motion for constructing a random chainable continuum. We show that 
this random continuum is indecomposable. We use our probabilistic model to define a Wiener-type measure on the space of all chainable continua. 
\end{abstract}

\maketitle

\section{Introduction}\label{S:intro} 

In this paper, we give a description of a random chainable continuum. (The relevant topological notions are defined in Section~\ref{Su:basde}.) 
In \cite{B}, Bing hypothesized that a certain topologically important continuum may be obtained as the intersection of 
a nested sequence of chains such that each 
chain is picked as a random refinement of the previous chain in a way similar to a random walk. 
(The continuum in question is called the pseudoarc; we give some more information on it in Section~\ref{Su:basde}.)
Bing's speculation was reiterated by Prajs in 
his talk \cite{P}, in which he also pointed out that the more basic question of finding a precise probabilistic model behind Bing's description is open. 
Here, we address this question by finding a probabilistic model for constructing a random chainable continuum. 
Furthermore, we show that a random chainable continuum is indecomposable. 

Our probabilistic model can be roughly described as follows. 
Let $B = (B(t), t \in \mathbb{R})$ be a standard Brownian motion with two-sided parameter. (In the remainder 
of the paper, we will call $B$ a two-sided Brownian motion. Its definition is given in Section~\ref{Su:basde}.) 
Let $B_n$, $n \ge 1$, be independent copies of $B$.
We give a procedure for constructing a sequence of 
non-degenerate time intervals $I_n$, $n\geq 1$, all containing $0$, such that $B_n$ maps $I_{n+1}$ onto $I_{n}$, so 
$B_n\res I_{n+ 1}$ can be thought of as a random refinement of $I_{n}$. 
Now, the procedure yields a sequence 
\begin{equation}\label{E:seqin} 
I_1\xleftarrow{B_1\res I_2} I_2\xleftarrow{B_2\res I_3} I_3\xleftarrow{B_3\res I_4} I_4 \xleftarrow{B_4\res I_5}\cdots , 
\end{equation}
which gives rise to the random limit continuum, namely the inverse limit 
\begin{equation}\label{E:rali}
\varprojlim_n\, (I_n,\, B_n\res I_{n+1}).
\end{equation}
A precise definition of the inverse limit is given in \eqref{E:inlim}. Here we only point out that it is a continuum that is a subset of 
${\mathbb R}^{\N}$ whose existence is guaranteed entirely by the diagram \eqref{E:seqin}. 

In Theorem~\ref{T:model}, we extract from $(B_n)$ such a sequence $(I_n)$ of non-degenerate intervals almost surely. 
This sequence is found in a canonical way without making arbitrary choices. The canonicity of the sequence is captured by the notion of 
a sequence of continuous functions from $\mathbb R$ 
to $\mathbb R$ determining a continuum; see Section~\ref{S:tap}. 
The extraction of $(I_n)$ from $(B_n)$ is done as follows. 
We fix an arbitrary non-degenerate time interval $J$ with $0\in J$. It turns out that the sequence of intervals in each row of the following matrix 
\begin{alignat}{5}
B_1(J),\; B_1\circ \, &B_2(J),\; B_1\circ\, &B_2\circ\, &B_3(J),\;   B_1\circ\, &B_2 \circ\, &B_3\circ B_4(J),\; \dots \notag\\
&B_2(J),\; &B_2\circ\, &B_3(J),  & B_2\circ\, &B_3\circ B_4(J),\; \dots \notag\\
&&&B_3(J),&&B_3\circ B_4(J),\;\dots \notag\\\notag
&&&\;\;\;\;\vdots 
\end{alignat}
converges almost surely to a non-degenerate interval that, importantly, does not depend on $J$. The limit interval in the $n$-th row of the matrix is the interval $I_n$. 

The immediate problem that now presents itself is to characterize the homeomorphism type of the limit continuum \eqref{E:rali} for the sequence $(I_n)$ chosen as above. 
In this direction, we show that that the limit continuum is indecomposable almost surely. 

Our proofs use work \cite{CK14} on iterated Brownian motion.

\section{Basic definitions}\label{Su:basde} 

Let $\N$ stand for the set of all positive integers, in particular, by this convention, $0\not\in \N$. 

By an {\bf interval} we understand a set of the form $\{ x\in {\mathbb R}\mid a\leq x\leq b\}$, where $a,b\in {\mathbb R}$, $a\leq b$, 
so it is a closed interval. The interval is called {\bf non-degenerate} if $a<b$. 
If $I, J$ are intervals, we write 
\begin{equation}\label{E:dist}
{\rm dist}(I,J) = \max ( |\min I - \min J|,\, |\max I-\max J|). 
\end{equation}
We note that ${\rm dist}(I,J)$ is the usual Hausdorff distance between the two compact sets $I$ and $J$. For a sequence of intervals 
$(I_n)$ and an interval $J$, we write $\lim_n I_n = J$ if 
$\lim_n {\rm dist}(I_n, J)=0$. 

A {\bf continuum} is a compact connected metric space. It is {\bf non-degenerate} if it has more than one point. 
A continuum $C$ is {\bf indecomposable} provided that 
whenever it is the union of two of its subcontinua $C_1$ and $C_2$, then $C_1=C$ or $C_2=C$. 
A continuum is {\bf hereditarily indecomposable} if each of its subcontinua 
is indecomposable. A continuum $C$ is called {\bf chainable} 
if for each $\epsilon>0$ there exists a continuous function $f\colon C\to [0,1]$ such that the preimages of points under $f$ 
have diameter less than $\epsilon$. 
Equivalently, a continuum is chainable if it is the inverse limit of a sequence of intervals with continuous bonding maps. (Inverse limits 
and bonding maps are defined below.) 
Another equivalent form of this notion says that a continuum is chainable when it is the intersection of a nested sequence of chains. 
An important chainable continuum is {\bf the pseudoarc}; for a definition see \cite{N}. 
By \cite{B51}, the pseudoarc is characterized as the unique chainable, 
hereditarily indecomposable continuum. Section~\ref{Su:bing} contains additional information on it. 
See also \cite{N} for more information on continua. 

The {\bf inverse limit} $\varprojlim_n\, (I_n,\, f_n\res I_{n+1})$, where $f_n : \mathbb{R} \to \mathbb{R}$ is a continuous function and 
$I_n \subseteq \mathbb{R}$ is an interval such that $f_n(I_{n+1}) = I_n$ for each $n \ge 1$, is defined as 
\begin{equation}\label{E:inlim}
\begin{split}
  \varprojlim_n\, (I_n,\, &f_n\res I_{n+1}) = \\
  &\left\{(x_1, x_2, \dots ) \in \mathbb{R}^{\mathbb{N}}\mid \forall n \in \mathbb{N} \left(x_n \in I_n \text{ and } f_n(x_{n+1}) = x_n\right)\right\}.
  \end{split}
\end{equation}
In such a system, each function $f_n$ is called a {\bf bonding map}. 

For basic notions concerning {\bf Brownian motion}, we refer the reader to \cite[Chapter 1]{MP10}. We only mention here, see 
\cite[Definition~1.1]{MP10}, that 
a Brownian motion is a measurable 
function $B\colon \Omega\times {\mathbb R}^+\to {\mathbb R}$, where $\Omega$ is 
a probability space, such that $B(\omega, 0)=0$ for all $\omega\in \Omega$; 
the function $B(\omega, \cdot)$ is continuous, on a measure $1$ set of $\omega\in \Omega$; 
for all $t_0<\cdots <t_k$, the random variables $B(\cdot, t_i)-B(\cdot, t_{i-1})$, $i=1, \dots, k$,  are independent; and 
for $s<t$, $B(\cdot, t)-B(\cdot, s)$ 
has the same distribution as $B(\cdot, t-s)$. 
Measurability of $B$, which is checked in \cite[Exercise 1.2]{MP10}, is understood with respect to the product measure, 
where $\mathbb R$ is equipped with Lebesgue measure. 
As is customary, most of the time, we suppress the first coordinate and write $B(t)$ for $B(\omega,t)$.

Using the notation in \cite{CK14}, let $B_+ = (B_+(t), t \ge 0)$ and $B_- = (B_-(t), t \ge 0)$ be 
two independent standard one-dimensional Brownian motions. 
We call the process $B$ defined by $B(t) = B_+(t)$ if $t \ge 0$ and $B(t) = B_-(-t)$ if $t < 0$ a {\bf two-sided Brownian motion}.

\section{The theorem and its proof}\label{S:tap} 

We consider sequences $\bar{f}= (f_n)$ of continuous functions $f_n\colon {\mathbb R}\to {\mathbb R}$ with $f_n(0)=0$. 
We recall here that Brownian motion's paths are almost surely continuous functions whose value is $0$ at time $0$.
We say that the sequence $\bar{f}$ {\bf determines a continuum} 
if there exists a sequence $(I_n)$ of intervals with $0\in I_n$ such that, for each non-degenerate interval $J$ with $0\in J$, we have 
\[
I_k = \lim_n \bigl( f_k\circ \cdots \circ f_{k+n}\bigr) (J),
\]
where the limit is taken with respect to the Hausdorff metric or, equivalently, the distance defined by \eqref{E:dist}.
Note that the sequence $(I_n)$ is determined solely by $\bar{f}$. 
Additionally, observe that 
\begin{equation}\label{E:sed} 
I_k = f_k(I_{k+1}),\hbox{ for each }k\geq 1. 
\end{equation}
To see this fact, fix $k \geq 1$ and an interval $J$ with $0 \in J$. Then for the intervals 
\[
J_n = (f_{k+1}\circ f_{k+2} \circ \dots \circ f_{k+n})(J), 
\]
we have 
$J_n \to I_{k+1}$ and $f_k(J_n) \to I_k$ as $n \to \infty$. The continuity of $f_k$ implies that $f_k(J_n) \to f_k(I_{k+1})$ as $n \to \infty$, proving 
\eqref{E:sed}. 
We say that $\bar{f}$ {\bf determines a non-degenerate continuum} 
if $I_n$ is non-degenerate for all but finitely many $n$.

Equation \eqref{E:sed} allows one to form the inverse 
limit 
\begin{equation}\label{e:K(f) definition}
K(\bar{f}) = \varprojlim_n\, (I_n, f_n\res I_{n+1}).
\end{equation}
This inverse limit is a chainable continuum, as witnessed by the maps 
\[
K(\bar{f}) \ni (x_1, x_2, \dots) \mapsto l_k(x_k)\in [0,1], 
\]
where $l_k$ is a linear bijection $l_k : I_k \to [0, 1]$. We call $K(\bar{f})$ the {\bf continuum determined by $\bar{f}$}. 
Note that not all sequences $\bar{f}$ 
determine a continuum; for example, the sequence 
$\bar{f}$ with $f_n= {\rm id}_{\mathbb R}$ does not; on the other hand, the sequence $\overline{g}$ given by 
$g_n(t) = \sin (\pi nt)$ does determine a non-degenerate continuum. 

Recall the definition of a two-sided Brownian motion from Section~\ref{Su:basde}.

\begin{theorem}\label{T:model}
\begin{enumerate}
\item[(i)] The sequence ${\bar B}= (B_n)_{n \ge 1}$ 
of independent two-sided Brownian motions determines a non-degenerate continuum with probability 1. 

\item[(ii)] The continuum determined by ${\bar B}$ is indecomposable with probability 1. 
\end{enumerate}
\end{theorem}

We spell out Theorem~\ref{T:model} by making explicit the dependence on the variable 
$\omega$ coming from 
the probability space $\Omega$ on which the $B_n$-s are defined. 
The conclusion of Theorem~\ref{T:model} asserts that the set of all $\omega\in \Omega$ such that 
\[
\begin{split}
&\hbox{the function }B_n(\omega, \cdot)\hbox{ is continuous for each }n\in\N, \hbox{ and}\\ 
&\hbox{the sequence }(B_n(\omega, \cdot))_{n\geq 1}\hbox{ determines a non-degenerate continuum, and}\\
&\hbox{the continuum determined by }(B_n(\omega, \cdot))_{n\geq 1}\hbox{ is indecomposable}
\end{split}
\]
has measure $1$.

The remainder of this section will be taken by the proof of the theorem above.

\begin{proof}[Proof of Theorem~\ref{T:model}]
We call an interval $I$ {\bf suitable} if it is non-degenerate and $0 \in I$. 

  \smallskip

\noindent {\em Proof of (i).} We denote by $W_n(t) = (B_1 \circ B_2 \circ \dots \circ B_n)(t)$ the composition of the first $n$ processes.

  
  
 We first state two claims, and show how the theorem follows from them. For an interval $I \subseteq \R$ and $c > 0$, let $c*I$ be the interval with the same center as $I$ and $\length(c*I) = c \length(I)$, where $\length(I)$ denotes the length of the interval $I$.
  
  \begin{claim}
    \label{c:there is a good k for delta}
    For every $\varepsilon > 0$ and $1> \delta > 0$, there exists a $k \in \N$ such that 
    \begin{equation}
      \label{e:short covers long}
      \prob\bigg((1 + \varepsilon) * W_k([-\delta, \delta]) \supseteq W_k\Big(\Big[\frac{-1}{\delta}, \frac{1}{\delta}\Big]\Big)\bigg) > 1 - \varepsilon.
    \end{equation}
  \end{claim}
  \begin{claim}
    \label{c:there is a good delta}
    For every $\varepsilon > 0$, there exists $1> \delta > 0$ such that for each $k \ge 1$, 
    \begin{equation}
      \label{e:decrease}
      \prob\bigg(\forall n \in \N \bigg( B_{k + n}\Big(\Big[\frac{-2^n}{\delta}, \frac{2^n}{\delta}\Big]\Big) \subseteq \Big[\frac{-2^{n - 1}}{\delta}, \frac{2^{n - 1}}{\delta}\Big]\bigg)\bigg) > 1 - \varepsilon,
    \end{equation}
    \begin{equation}
      \label{e:right increase}
      \prob\bigg(\forall n \in \N \bigg( B_{k + n}\Big(\Big[0, \frac{\delta}{2^n}\Big]\Big) \supseteq \Big[\frac{-\delta}{2^{n - 1}}, \frac{\delta}{2^{n - 1}}\Big]\bigg)\bigg) > 1 - \varepsilon,
    \end{equation}
    \begin{equation}
      \label{e:left increase}
      \prob\bigg(\forall n \in \N \bigg( B_{k + n}\Big(\Big[\frac{-\delta}{2^n}, 0\Big]\Big) \supseteq \Big[\frac{-\delta}{2^{n - 1}}, \frac{\delta}{2^{n - 1}}\Big]\bigg)\bigg) > 1 - \varepsilon.
    \end{equation}
  \end{claim}
  
Since the sequences of random variables $(B_{k+n})_n$ have the same distribution for all $k$, it is enough to show that with probability 1, the limit $\lim_n W_n(J)$ exists and is the same non-degenerate interval $I_1$ for each suitable interval $J$. 
  
  Let us fix an $\varepsilon > 0$ towards proving an $\varepsilon$ approximation of the above statement. Let $\delta(\varepsilon) = \delta > 0$, $\delta < 1$, 
  be given by Claim \ref{c:there is a good delta} for $\varepsilon$. Then for each $k \in \N$ Claim \ref{c:there is a good delta} \eqref{e:decrease} implies that with probability greater than $1 - \varepsilon$ for each $n \in \N$ 
  \begin{align*}
    B_{k+1} \circ \dots \circ B_{k+n}\left(\left[\frac{-2^n}{\delta}, \frac{2^n}{\delta}\right]\right) 
    & \subseteq 
    B_{k+1} \circ \dots \circ B_{k+n-1}\left(\left[\frac{-2^{n-1}}{\delta}, \frac{2^{n-1}}{\delta}\right]\right) 
    \\ & \subseteq
    B_{k+1} \circ \dots \circ B_{k+n-2}\left(\left[\frac{-2^{n-2}}{\delta}, \frac{2^{n-2}}{\delta}\right]\right) 
  \end{align*}
  and so for each $k \in \N$, with probability greater than $1 - \varepsilon$,
  \begin{equation}
    \label{e:iterates end up in large interval}
    B_{k+1} \circ \dots \circ B_{k+n}\left(\left[\frac{-2^n}{\delta}, \frac{2^n}{\delta}\right]\right) \subseteq \left[\frac{-1}{\delta}, \frac{1}{\delta}\right]
  \end{equation}
  holds for all $n \in \N$.
  With a similar argument, Claim \ref{c:there is a good delta} \eqref{e:right increase} and \eqref{e:left increase} imply that for each $k \in \N$ with probability larger than $1 - 2\varepsilon$, 
  \begin{equation}
  \label{e:iterates cover short interval}
  \begin{split}
    B_{k+1} \circ \dots \circ B_{k+n}\left(\left[0, \frac{\delta}{2^n}\right]\right) \supseteq \left[{\delta}, {\delta}\right] \\
    B_{k+1} \circ \dots \circ B_{k+n}\left(\left[\frac{-\delta}{2^n}, 0\right]\right) \supseteq \left[{\delta}, {\delta}\right]
  \end{split}
  \end{equation}
  both hold for all $n \in \N$. Therefore with probability greater than $1-3\varepsilon$, all three equations hold for each $n \in \N$.

  Let $\kappa$ be the $k$ chosen Claim \ref{c:there is a good k for delta} for the given
  $\varepsilon$ and $\delta$ and put 
  \[
  I(\varepsilon) = W_\kappa([-\delta, \delta]).  \]
  Using \eqref{e:short covers long}, \eqref{e:iterates end up in large interval} and \eqref{e:iterates cover short interval} with $k = \kappa$, with probability greater than $1 - 4\varepsilon$, 
  \begin{equation}
  \label{e:main equations}
  \begin{split}
    W_\kappa \circ B_{\kappa +1} \circ \dots \circ B_{\kappa + n}\left(\left[\frac{-2^n}{\delta}, \frac{2^n}{\delta}\right]\right) &\subseteq W_\kappa\left(\left[\frac{-1}{\delta}, \frac{1}{\delta}\right]\right) \subseteq (1+\varepsilon)*I(\varepsilon), \\
    W_\kappa \circ B_{\kappa +1} \circ \dots \circ B_{\kappa + n}\left(\left[0, \frac{\delta}{2^n}\right]\right) &\supseteq I(\varepsilon), \\
    W_\kappa \circ B_{\kappa +1} \circ \dots \circ B_{\kappa + n}\left(\left[\frac{-\delta}{2^n}, 0\right]\right) &\supseteq I(\varepsilon), \\
  \end{split}
  \end{equation}
  all hold for each $n \in \N$. 
  
  For each suitable interval $J$ there is an index $n_0(J) > 0$ such that for each $n \ge n_0(J)$ 
  \begin{equation}\label{E:varinc} 
    J \subseteq [-2^n/\delta, 2^n/\delta]\;\hbox{ and either }\;[0, \delta/2^n] \subseteq J\;\hbox{ or }\;[-\delta/2^n, 0] \subseteq J. 
  \end{equation} 
  Thus, using \eqref{e:main equations}, we see that with probability greater than $1 - 4\varepsilon$, 
  \begin{equation}\label{e:W_kappa+n(J) is sandwiched}
    I(\varepsilon) \subseteq W_{\kappa + n}(J) \subseteq (1 + \varepsilon) * I(\varepsilon), 
  \end{equation}
  for all suitable $J$ and large enough $n$ depending on $J$.
  
  We now apply the conclusion above for $\varepsilon, \varepsilon/2, \varepsilon/4, \dots$ simultaneously, to obtain that with probability at least $1 - 8 \varepsilon$, \eqref{e:W_kappa+n(J) is sandwiched} holds for $\frac{\varepsilon}{2^j}$ for each $j \ge 0$. Now suppose that \eqref{e:W_kappa+n(J) is sandwiched} holds for each $\frac{\varepsilon}{2^j}$. It is enough to show that in this case, $\lim_{n \to\infty} W_n(J)$ exists and is the same non-degenerate interval for all suitable $J$. To see this, note that $I(\varepsilon) \subseteq (1 + \frac{\varepsilon}{2^j}) * I(\frac{\varepsilon}{2^j})$ and $I(\frac{\varepsilon}{2^j}) \subseteq (1 + \varepsilon) * I(\varepsilon)$, hence there is a universal upper and lower bound on the length of the intervals $I(\frac{\varepsilon}{2^j})$. Similarly, we can write $I(\frac{\varepsilon}{2^l}) \subseteq (1 + \frac{\varepsilon}{2^j}) * I(\frac{\varepsilon}{2^j})$ and $I(\frac{\varepsilon}{2^j}) \subseteq (1 + \frac{\varepsilon}{2^l}) * I(\frac{\varepsilon}{2^l})$ for each $j$ and $l$, hence the sequence of intervals $(I(\frac{\varepsilon}{2^j}))_{j \ge 0}$ is Cauchy in the complete Hausdorff metric. Thus, it converges to some non-degenerate interval $I_1$ as $j \to \infty$. As the distance of $I(\frac{\varepsilon}{2^j})$ and $(1 + \frac{\varepsilon}{2^j}) * I(\frac{\varepsilon}{2^j})$ converges to $0$ as $j \to \infty$, the sequence $((1 + \frac{\varepsilon}{2^j}) * I(\frac{\varepsilon}{2^j}))_{j \ge 0}$ converges to $I_1$ as well. Using \eqref{e:W_kappa+n(J) is sandwiched}, it follows that $W_{\kappa + n}(J)$ converges to $I_1$ for each suitable interval $J$. Therefore, it remains to prove the two claims. 
  
  Before we start, let us denote by $M(t) = \sup_{x \in [0, t]} B(x)$ the maximum of a Brownian motion on the interval $[0, t]$. 
  We recall from \cite[Theorem 2.21]{MP10} the formula
    \begin{equation}
      \label{e:max}
      \prob(M(t) > a) = 2\prob(B(t) > a),\, \text{ for $a > 0$,}
    \end{equation}
    from which we also obtain 
     \begin{equation}
      \label{E:max}
      \prob(M(t) \leq a) = 1 - 2\prob(B(t) > a )= 2\prob(0\leq B(t) \leq a),\, \text{ for $a > 0$.}
    \end{equation}

  \begin{proof}[Proof of Claim \ref{c:there is a good k for delta}]
    We start with a formula for the distribution of the length of the image of an interval under $W_n$.
    We denote the oscillation of $W_k$ on the interval $[0, t]$ by  $$\Delta_k(t) = \length(W_k([0, t])) = \sup_{0 \le x \le y \le t}|W_k(x) - W_k(y)|  = \sup_{0 \le x \le t}W_k(x) - \inf_{0 \le x \le t} W_k(x)$$ and let $D_1, D_2, \dots$ be independent, identically distributed copies of $\Delta_1(1)$. Then by \cite[Lemma 9]{CK14}, 
    \begin{equation}
      \label{e:dist of length}
      \Delta_k(t) \equaldist  t^{2^{-k}} \prod_{i = 1}^k D_i^{2^{-(i - 1)}},
    \end{equation}
    where $\equaldist $ means equality in distribution. 
    
    We show that the expected value of $|\log(\Delta_k(t))|$ is finite. Using \eqref{e:dist of length}, we get 
    \begin{equation}\label{E:delo}
    \expval \log(\Delta_k(t)) = 2^{-k} \log t + \sum_{i = 1}^k 2^{-(i-1)}\expval \log(D_i).
    \end{equation}
    Therefore, it is enough to show that $\expval |\log(\Delta_1(1))|$ is finite. We compute 
    \begin{align*}
      \expval &|\log(\Delta_1(1))| = \int_{0}^{\infty} \prob(\log(\Delta_1(1)) > x) dx + \int_{-\infty}^{0} \prob(\log(\Delta_1(1)) \le x) dx \\ 
      &= \int_{0}^{\infty} \prob(\Delta_1(1) > e^x) dx + \int_{-\infty}^{0} \prob(\Delta_1(1) \le e^x) dx \\ 
      &\le 2\int_{0}^{\infty} \prob(M(1) > e^x/2) dx + \int_{-\infty}^{0} \prob(M(1) \le e^x) dx \\ 
      &= 4\int_{0}^{\infty} \prob(B(1) > e^x/2) dx + 2 \int_{-\infty}^{0} \prob(0\leq B(1)\leq e^x) dx\\ 
      &= 4\int_{0}^{\infty} \frac{1}{\sqrt{2\pi}}\int_{e^x/2}^\infty e^{-u^2/2}du dx + 2\int_{-\infty}^{0} \frac{1}{\sqrt{2\pi}}\int_{0}^{e^x} e^{-u^2/2} du dx \\
      &\le 4\int_{0}^{\infty} \frac{1}{\sqrt{2\pi}}\left(\int_{e^x/2}^\infty e^{-u/4}du \right) dx + 2\int_{-\infty}^{0} \frac{1}{\sqrt{2\pi}}\left( \int_{0}^{e^x} 1 du \right)dx <\infty
    \end{align*}
    where we used \eqref{e:max} and \eqref{E:max} for the third equality and $-u^2/2\leq -u/4$ holding for $u\geq 1/2$ for the second inequality.
    
    Using \eqref{E:delo}, finiteness of $\expval \log(\Delta_k(t))$, and the fact that the distribution of the length of the image of an interval under $W_k$ only depends on the length of the interval, we get 
    \[
    \expval \log(\length(W_k([-1/\delta, 1/\delta]))) - \expval \log(\length(W_k([-\delta, \delta]))) = \frac{\log(1/\delta^2)}{2^k}.
    \]
    Now we choose $k$ large enough so that the difference is smaller than $\varepsilon\log(1 + \varepsilon/2)$. Then Markov's inequality implies that with probability at least $1 - \varepsilon$, $$\log(\length(W_k([-1/\delta, 1/\delta]))) - \log(\length(W_k([-\delta, \delta]))) \le \log(1 + \varepsilon/2).$$
    Therefore, with probability at least $1 - \varepsilon$, \begin{equation}
      \label{e:ratio of lengths}
      \frac{\length(W_k([-1/\delta, 1/\delta]))}{\length(W_k([-\delta, \delta]))} \le 1 + \frac{\varepsilon}{2}.
    \end{equation} 
    The fact that $W_k([-1/\delta, 1/\delta]) \supseteq W_k([-\delta, \delta])$ and \eqref{e:ratio of lengths} imply $$(1+\varepsilon) * W_k([-\delta, \delta]) \subseteq W_k([-1/\delta, 1/\delta]),$$
    and the claim follows.
    \end{proof}
  \begin{proof}[Proof of Claim \ref{c:there is a good delta}]
    We find a suitable $\delta$ in the form of $1/2^n$. The proof consists of calculations using \eqref{e:max} and \eqref{E:max}. Let $N$ be a standard normal random variable. For $n \in \N$, we have 
    \begin{align*}
      \prob&\big(B([-2^n, 2^n]) \not\subseteq [-2^{n - 1}, 2^{n - 1}]\big) \le 4\prob\big(M(2^n) > 2^{n - 1}\big)
      = 8\prob\big(B(2^n) > 2^{n - 1}\big) \\ &= 8\prob\big(2^{\frac{n}{2}}B(1) > 2^{n-1}\big) = 8\prob \big(N > 2^{\frac{n}{2} - 1}\big) \le 2^{-\frac{n}{2} + 4},
    \end{align*}
    using \eqref{e:max}, the scaling invariance of the Brownian motion (see e.g.~\cite[Lemma 1.7]{MP10}) and the inequality $\prob(N > x) \le \frac{1}{x}$. We also get 
    \begin{align*}
      \prob&\bigg(B\Big(\Big[0, \frac{1}{2^n}\Big]\Big) \not \supseteq \Big[-\frac{1}{2^{n - 1}}, \frac{1}{2^{n - 1}}\Big]\bigg) \le 2\prob\bigg(M\Big(\frac{1}{2^n}\Big) < 
      \frac{1}{2^{n - 1}}\bigg) \\
      &= 4 \prob\bigg(0 \le B\Big(\frac{1}{2^n}\Big) < \frac{1}{2^{n - 1}}\bigg) 
      = \frac{4\cdot 2^{n/2}}{\sqrt{2\pi}}\int_0^{1/2^{n - 1}} e^{-u^22^n/2} du \\
      &\le \frac{4\cdot 2^{n/2}}{2^{n - 1}\sqrt{2\pi}} \le 2^{-n/2 + 2},
    \end{align*}
    where \eqref{E:max} is applied for the first equality. And similarly, we obtain 
    \begin{align*}
      \prob\bigg(B\Big(\Big[-\frac{1}{2^n}, 0\Big]\Big) \not \supseteq \Big[-\frac{1}{2^{n - 1}}, \frac{1}{2^{n - 1}}\Big]\bigg) \le 2^{-n/2 + 2}.
    \end{align*}
    These probabilities are summable, hence for a given $\varepsilon > 0$, for large enough $n_0$, the sum of them for $n \ge n_0$ is less than 
    $\varepsilon$. By setting $\delta = 1/2^{n_0}$, the claim follows from the fact that the probabilities are independent of $k$. 
\end{proof}

\smallskip

\noindent {\em Proof of (ii).} We keep our notation from the proof of point (i). For a suitable interval $I$, let 
\[
w(I) = \min ( \max (I), -\min(I)). 
\]

\begin{claim}\label{C:lims}
$\limsup_n w(I_n)>0$ with probability 1. 
\end{claim}

Assuming the claim, we show point (ii). Let $K= K({\bar B})$ be the continuum determined by ${\bar B} = (B_n)_{n \in \N}$, as defined in \eqref{e:K(f) definition}. Assume 
\begin{equation}\label{E:uni}
K= L^1\cup L^2
\end{equation}
with $L^1$ and $L^2$ being continua. We aim to show that $K=L^1$ or $K=L^2$. Let $J^1_n = \{x_n : (x_1, x_2, \dots) \in L^1\}$, and similarly, $J^2_n = \{x_n : (x_1, x_2, \dots) \in L^2\}$ for each $n$. From compactness and connectedness of $L^1$ and $L^2$, 
we get that $J^1_n$ and $J^2_n$ are compact intervals for each $n$. Moreover, since $L^1, L^2 \subseteq K$, and $K$ is an inverse limit with bonding maps $B_n$, we see that 
$B_n$ maps $J^1_{n+1}$ onto $J_n^1$ and $J^2_{n+1}$ onto $J^2_n$. Together with equality \eqref{E:uni}, we conclude that 
\begin{equation}\label{E:var}
I_n = J^1_n\cup J^2_n,\; B_{n}(J^1_{n+1}) = J^1_n,\hbox{ and } B_{n}(J^2_{n+1}) = J^2_n,
\end{equation}
and 
\begin{equation}\label{E:pal}
L^1 = \varprojlim_n\, (J^1_n,\, B_n\res J^1_{n+1})\;\hbox{ and }\; L^2 = \varprojlim_n\, (J^2_n,\, B_n\res J^2_{n+1}).
\end{equation}
From Claim~\ref{C:lims}, we obtain $d>0$ and an infinite set $X\subseteq {\mathbb N}$ such that 
\begin{equation}\label{E:hep}
\max I_{n} >d\hbox{ and } \min I_{n}<-d\hbox{ for all }n\in X. 
\end{equation} 
Now, noting that $0 \in I_n$ for each $n$, $\{\min I_n, 0, \max I_n\} \in J^1_n \cup J^2_n$ by the first inequality in \eqref{E:var}. Hence, for some $i_0 \in \{1, 2\}$ we have that $0, \max I_n \in J_n^{i_0}$ or $0, \min I_n \in J^{i_0}_n$. We then get $i_0\in \{ 1, 2\}$ and an infinite set $Y\subseteq X$ such that 
\begin{equation}\label{E:alt}
0,\, \max I_{n}\in J^{i_0}_{n} \hbox{ for all }n\in Y\;\hbox{ or }\; 0,\, \min I_{n}\in J^{i_0}_{n} \hbox{ for all }n\in Y.
\end{equation}
To fix attention, assume that $i_0=1$ and that, in \eqref{E:alt}, the first alternative holds. Then, using \eqref{E:hep}, we see that 
\begin{equation}\label{E:mos}
[0, d]\subseteq J^1_n\hbox{ for all }n\in Y. 
\end{equation}
Now, for each $m$, in the sense of the Hausdorff metric,
\[
I_m = \lim_{n>m,\, n\in Y} \bigl( B_{m}\circ \cdots \circ B_n\bigr)([0,d]) \subseteq \bigcup_n \bigl( B_{m}\circ \cdots \circ B_n\bigr)(J^1_n)= J^1_m,
\]
where the first equality holds by (i) of the theorem, the inclusion holds by \eqref{E:mos}, and the last equality holds by \eqref{E:var}. 
So we proved that $I_m=J^1_m$ holds for all $m$, which gives $K=L^1$ by \eqref{E:pal}, as required. 

It remains to show the claim. 

\begin{proof}[Proof of Claim~\ref{C:lims}.] Let 
\[
a_n = w(I_n).
\]
Then $(a_n)$ is a sequence of identically distributed random variables. 
We have $a_1>0$ with probability 1. Indeed, the random variables $\max I_1$ and $-\min I_1$ have the same distribution by the symmetry of the formula 
\[
I_1 = \lim_n W_n([-1,1]). 
\]
Thus, since $I_1$ is non-degenerate with probability 1, we have that both $\max I_1>0$ and $-\min I_1>0$ hold with probability 1, which gives 
${\mathbb P}(a_1>0) =1$. 
It follows that there exists $d>0$ such that 
\begin{equation}\notag
{\mathbb P}(a_1<d)<1-2\epsilon
\end{equation}
for small enough $\epsilon>0$. Since the sequence $(a_n)$ is identically distributed, we get that, for small enough $\epsilon>0$, 
\begin{equation}\label{E:and} 
\forall n\; {\mathbb P}(a_n<d)<1-2\epsilon. 
\end{equation}

Find, a sequence $1=n_0<n_1<n_2<\cdots$ such that 
\begin{equation}\label{E:ns} 
{\mathbb P}\Bigl( {\rm dist}\bigl(I_{n_k}, B_{n_k} \circ \cdots \circ B_{n_{k+1}-1}([0,1])\bigr) <\frac{d}{3}\Bigr) >1-\frac{\epsilon}{2^{k+1}}.
\end{equation}
Define
\[
b_k= w\bigl(B_{n_k} \circ \cdots \circ B_{n_{k+1}-1}([0,1])\bigr). 
\]
So, $(b_k)$ is a sequence of independent random variables.
By \eqref{E:ns}, 
\begin{equation}\label{E:pdd}
{\mathbb P}\bigl( \forall k\;  a_{n_k}> b_k - \frac{d}{3}\bigr) > 1-\sum_k \frac{\epsilon}{2^{k+1}} = 1-\epsilon
\end{equation} 
and, by \eqref{E:ns} and \eqref{E:and}, for each $k$ 
\begin{equation}\label{E:pri}
{\mathbb P}\bigl(b_k<\frac{2d}{3}\bigr) \leq {\mathbb P}\bigl(a_{n_k} < d\bigr) + \epsilon < 1-\epsilon. 
\end{equation}
By independence of $(b_k)$ it follows from \eqref{E:pri} that 
\[
{\mathbb P}\bigl(b_k<\frac{2d}{3}\hbox{ for all but finitely many }k\bigr) = 0, 
\]
and, therefore, 
\[
{\mathbb P}\bigl(b_k\geq \frac{2d}{3}\hbox{ for infinitely many }k\bigr) =1. 
\]
Thus, by \eqref{E:pdd}, 
\[
{\mathbb P}\bigl(a_{n_k}\geq \frac{d}{3}\hbox{ for infinitely many }k\bigr) >1-\epsilon. 
\]
It follows that 
\[
{\mathbb P}\bigl(\limsup_n a_n \geq \frac{d}{3}\bigr)>1-\epsilon.
\]
Since $\epsilon>0$ can be made arbitrarily small with the fixed $d$, the claim is proved. 
\end{proof} 
This finishes the proof of the theorem. 
\end{proof}

\section{Further observations}

\setcounter{claim}{0}

We make here a few comments on the main construction of this paper. We also present some alternative probabilistic models for 
a random continuum. 
Let ${\mathcal C}(X)$ be the space of all continua that are subsets of a Polish space $X$. Equip ${\mathcal C}(X)$ 
with the topology induced by the Hausdorff metric making ${\mathcal C}(X)$ into a Polish space; see \cite[3.12.27(b,g), 6.3.22(b)]{E}. 

\subsection{A Wiener-type measure on the space of continua} 

We rephrase Theorem~\ref{T:model} in terms of a Borel probability measure defined on the space of all subcontinua 
of ${\mathbb R}^{\N}$. 

Theorem~\ref{T:model}(i) allows us to define a Borel probability measure $\beta$ on ${\mathcal C}({\mathbb R}^{\N})$, 
which is a version of Wiener measure. We describe now how $\beta$ is constructed. 
Let ${\bar B}= (B_n)_{n\geq 1}$ be a sequence of independent two-sided Brownian motions as in Theorem~\ref{T:model}. 
For $\omega\in \Omega$, we write 
\begin{equation}\label{E:bom}
{\bar B}(\omega)
\end{equation}
for the sequence of functions $\left(B_n(\omega, \cdot)\right)_{n\geq 1}$. Theorem~\ref{T:model}(i) asserts that the set 
\[
\begin{split}
\Omega_0 = \{ \omega\in \Omega \mid\, &B_n(\omega, \cdot)\hbox{ is continuous for each }n\in\N,\hbox{ and }\\ 
&{\bar B}(\omega)\hbox{ determines a continuum}\}
\end{split}
\]
has measure $1$.  
Recall definition \eqref{e:K(f) definition}, and consider the following function defined on $\Omega_0$: 
\begin{equation}\label{E:trans}
\Omega_0\ni \omega\mapsto K\bigl({\bar B}(\omega)\bigr)\in {\mathcal C}({\mathbb R}^{\N}).
\end{equation}

\begin{lemma}
The function defined by formula \eqref{E:trans} is measurable on $\Omega_0$. 
\end{lemma} 

\begin{proof} Let 
\[
B\colon \Omega\times {\mathbb R}\to {\mathbb R}. 
\]
be a two-sided Brownian motion. For $\omega\in \Omega$ and $t\in {\mathbb R}$, we write 
\[
B_\omega\;\hbox{ and }\; B^t
\]
for the functions $B(\omega, \cdot)$ and $B(\cdot, t)$, respectively.

Let $C({\mathbb R}, {\mathbb R})$ be the space of continuous functions from $\mathbb R$ to $\mathbb R$ equipped with 
the topology of uniform convergence on compact sets. It is a separable metric space.

Let 
\[
\Omega_1=\{ \omega\in \Omega\mid B_\omega \hbox{ is continuous}\}.
\]
Then $\Omega_1$ has measure $1$.  
\begin{claim}\label{Cl:one} 
The function 
\[
\Omega_1\ni\omega \mapsto B_\omega\in C({\mathbb R}, {\mathbb R})
\]
is measurable.
\end{claim}

\noindent {\em Proof of Claim~\ref{Cl:one}.} Let 
\[
T= \{ t\in {\mathbb R}\mid B^t  \hbox{ is measurable}\}.
\]
First we prove that $T={\mathbb R}$. 
Since $B$ is a measurable function, by Fubini's Theorem, $T$ has full Lebesgue measure, in particular, it is dense in 
$\mathbb R$. 
Let $t\in {\mathbb R}$ be arbitrary, and let $t_k$, $k\in \N$, be a sequence of elements of $T$ convergent to $t$. Then, by the definition of 
$\Omega_1$, we have that 
\[
B^{t_k}\res \Omega_1\to B^t\res \Omega_1 \hbox{ pointwise}, \hbox{ as }k\to\infty.
\]
Thus, $B^t\res \Omega_1$ is measurable, and, since $\Omega_1$ has measure $1$, $B^t$ is measurable, that is, $t\in T$.
Now, $T={\mathbb R}$ implies that, for all real numbers $t$ and $r\leq s$, the set 
\begin{equation}\label{E:onet}
\{ \omega\in \Omega\mid r\leq B^t(\omega)\leq s\}\hbox{ is measurable.}
\end{equation}

To prove the claim, it suffices to show that, for a compact set $K\subseteq {\mathbb R}$ and continuous functions 
$f,g\colon K\to {\mathbb R}$ with $f\leq g$, the set 
\begin{equation}\label{E:metw}
\{ \omega\in \Omega_1\mid f(t) \leq  B_\omega(t) \leq  g(t)\}
\end{equation}
is measurable. Let $Q\subseteq K$ be countable and dense in $K$. Then, by continuity of $f$, $g$, and $B_\omega$ for 
$\omega\in \Omega_1$, the set \eqref{E:metw} is equal to 
\[
\bigcap_{t\in Q} \{ \omega\in \Omega_1\mid f(t) \leq B^t(\omega)\leq g(t)\},
\]
which is measurable by \eqref{E:onet}, and the claim follows. 

\smallskip

Let ${\bar B}= (B_n)$ be a sequence of independent two-sided Brownian motions on $\Omega$. Then 
\[
\Omega_2  =\{ \omega\in \Omega\mid (B_n)_\omega \hbox{ is continuous, for each }n\in {\mathbb N}\}
\]
has measure $1$ as the intersection of countably many sets of measure $1$. Recall the definition 
\eqref{E:bom} of ${\bar B}(\omega)$ and note that the following claim is an immediate consequence of 
Claim~\ref{Cl:one}. (Independence of the $B_n$-s, $n\in \N$, is not used here.) 
\begin{claim}\label{Cl:two}
The function 
\[
\Omega_2 \ni\omega \mapsto {\bar B}(\omega) \in C({\mathbb R}, {\mathbb R})^\N 
\]
is measurable.
\end{claim} 

Define now 
\[
D = \{ {\bar f}\in C({\mathbb R}, {\mathbb R})^\N\mid {\bar f} \hbox{ determines a continuum}\}. 
\]
We adopt the convention that ${\bar f}=(f_n)$. 

Recall that a function from a metric space to a metric space is called {\bf Borel} if preimages of open sets belong to the $\sigma$-algebra 
of Borel sets; more information on Borel functions can be found in \cite[Section~11C]{Ke}. 
\begin{claim}\label{Cl:three}
The function 
\[
D\ni {\bar f} \mapsto K({\bar f}) \in {\mathcal C}({\mathbb R}^\N)
\]
is Borel. 
\end{claim}

\noindent{\em Proof of Claim~\ref{Cl:three}.}  
Let
\[
E= \{ \left({\bar f}, (I_n)\right)\in C({\mathbb R}, {\mathbb R})^\N\times {\mathcal C}({\mathbb R})^\N \mid f_n(I_{n+1})= I_n\hbox{ for each }n\in\N\}. 
\]
It is easy to check that the function 
\begin{equation}\label{E:lala}
E\ni \left({\bar f}, (I_n)\right) \mapsto \varprojlim (I_n, f_n\res I_{n+1}) \in {\mathcal C}({\mathbb R}^\N)
\end{equation}
is continuous. 

For ${\bar f}\in D$, for each $n\in\N$, there exists an interval $I_n({\bar f})$ such that 
\begin{equation}\label{E:lim}
I_n({\bar f})=\lim_k (f_n\circ \cdots \circ f_{n+k})([-1,1]).
\end{equation}
For a fixed $k\in \N$, the function 
\begin{equation}\label{E:cont}
D\ni {\bar f}\mapsto (f_n\circ \cdots \circ f_{n+k})([-1,1]) \in {\mathcal C}({\mathbb R})
\end{equation}
is clearly continuous hence, by \eqref{E:lim}, the function 
\[
D\ni {\bar f}\mapsto I_n({\bar f}) \in {\mathcal C}({\mathbb R})
\]
is Borel as a pointwise limit of a sequence of continuous functions \eqref{E:cont}. 
It follows that the function 
\[
D\ni {\bar f}\mapsto \left(I_n({\bar f})\right) \in {\mathcal C}({\mathbb R})^\N
\]
is Borel, from which we get that 
\begin{equation}\label{E:rara}
D\ni {\bar f}\mapsto \left( {\bar f}, \left(I_n({\bar f})\right)\right) \in E 
\end{equation}
is Borel as well. Thus, the function in the statement of the claim is Borel as the composition of \eqref{E:rara} with \eqref{E:lala}, and the claim follows.

\smallskip

Now, clearly $\Omega_0\subseteq \Omega_2$ and, for each 
$\omega\in \Omega_0$, ${\bar B}(\omega)$ is an element of $D$. Thus, the function in \eqref{E:trans} is measurable 
as the composition of the functions from Claims~\ref{Cl:two} and \ref{Cl:three}. 
\end{proof}

The above lemma allows us to transfer the probability measure on $\Omega$ to ${\mathcal C}({\mathbb R}^{\N})$ using function \eqref{E:trans}. 
We denote this transferred measure by $\beta$. Theorem~\ref{T:model}(ii)  
is equivalent to asserting that the set 
\[
\{ K\in  {\mathcal C}({\mathbb R}^{\N})\mid K\hbox{ is a non-degenerate chainable indecomposable continuum}\}
\]
is of full measure with respect to $\beta$. 

It would be interesting to see if $\beta$ is in some sense canonical. 

\subsection{Bing's question}\label{Su:bing}

Bing's question from \cite{B} can now be stated precisely. With the notation from Theorem~\ref{T:model}, 
is it the case that $K({\bar B})$ is the pseudoarc with probability 1? Or, equivalently, in terms of 
the measure $\beta$ defined above, is it the case that 
\[
\beta\bigl( \{ K\in  {\mathcal C}({\mathbb R}^{\N})\mid K\hbox{ is the pseudoarc}\}\bigr) = 1?
\]
By Bing's characterization of the pseudoarc \cite{B51}, the above questions are equivalent to asking 
whether $K({\bar B})$ is hereditarily indecomposable with probability 1.

In the topological context, as opposed to the measure theoretic context considered in this paper, prevalence of the pseudoarc has been known for a while. 
By \cite{B51}, the set of continua homeomorphic to the pseudoarc is comeager in the space of continua ${\mathcal C}([0,1]^{\N})$. 
Similarly for inverse limits. Let $C_s([0,1],[0,1])$ be the space of all continuous surjections from $[0,1]$ to itself. It is a Polish space when equipped with the uniform 
convergence topology. By \cite{BKU00}, the set of sequences $(f_n)\in C_s([0,1], [0,1])^{\N}$ such that $\varprojlim_n ([0,1], f_n)$ 
is homeomorphic to the pseudoarc is comeager in $C_s([0,1], [0,1])^{\N}$.

\subsection{Comments on other models for a random continuum}

We present here some other possible ways of modeling a random chainable continuum. At this point, we find them less interesting than the way studied in this paper as 
they involve certain arbitrary choices and do not involve unaltered Brownian motion (or unaltered random walk). 

{\bf 1.} One considers a sequence of independent Brownian motions $(B_n)$ and modifies them to reflected Brownian motions $(|B_n|)$. (For the reflected 
Brownian motion see \cite[Section~2.3]{MP10}.) Then one 
chooses a sequence of random variables $(T_n)$ so that $0<T_n<\infty$ and $|B_n|([0, T_{n+1}]) = [0,T_n]$ almost surely. Finally, one defines the random continuum 
\[
\varprojlim_n\, ([0, T_n],\, |B_n|\res [0, T_{n+1}]). 
\]

{\bf 2.} We recast the construction from point 1 above making it combinatorial. This is done by using 
the random walk, instead of the Brownian motion, and the point of view from \cite{IS}. We make a concrete choice for the sequence of random variables $(T_n)$ 
and provide some detailed arguments. 

For $n\in \N$, let $[n]=\{ 1, \dots , n\}$. 
A {\bf walk} is a function $f\colon [m]\to [n]$, $m,n\in {\mathbb N}$, $m,n\geq 1$, such that $f(1)=1$, $f$ is surjective, and 
\[
|f(x)-f(x+1)|\leq 1
\]
for all $x\in [m-1]$. 

We produce a sequence $(k_n)$ of elements of $\N$ and a sequence of walks $$f_n\colon [k_{n+1}]\to [k_n].$$ 
Set $k_0=2$. Assume $k_n$ is given. We define $f_{n}$ by setting $f_{n}(1)=1$ and requiring that if 
$f_{n}(2x - 1)$ is given, then $f_n(2x) = f_n(2x - 1)$, and 
\\ if $1<f_{n}(2x)<k_n$, then 
\begin{enumerate}
\item[---]$f_{n}(2x+1) = f_{n}(2x)+1$ with probability $1/2$, 

\item[---]$f_{n}(2x+1) = f_{n}(2x)-1$ with 
probability $1/2$; 
\end{enumerate} 

\noindent if $f_{n}(2x)=1$, then $f_{n}(2x+1) = 2$;
\\ if $f_{n}(2x)=k_n$, then $f_{n}(2x+1) = k_n-1$.  

We stop this process defining $f_{n}$ when we reach $x_0$ such that $f_{n}(x_0)= k_n$. 
We let $k_{n+1}=x_0$. This stopping procedure is somewhat arbitrary and can probably be modified without changing 
the fundamental properties of the model. So each $f_n$ is a truncated reflected random walk on $[k_n]$. 

The three claims below and definition \eqref{E:der} give a description of the probabilistic model. 

The following claim is a consequence of, for example, \cite[Theorem~5.4]{MP10}.

\begin{claim*} With probability 1, the sequence $(k_n)$ is defined.
\end{claim*}

We view $[m]$, for $m\in \N$, $m\geq 1$, as a finite discrete topological space with $m$ points. 
Consider the inverse limit $\varprojlim_n ([k_n], f_n)$ of topological spaces. Define the following relation $R$ on it. For $(x_n), (y_n)\in \varprojlim_n ([k_n], f_n)$, 
let 
\begin{equation}\label{E:der}
(x_n)\,R\,(y_n) \Leftrightarrow \forall n\, |x_n-y_n|\leq 1. 
\end{equation}

\begin{claim*} With probability 1, $R$ is an equivalence relation that is 
compact when seen as a subset of the product $\varprojlim_n ([k_n], f_n)\times \varprojlim_n ([k_n], f_n)$. Each equivalence class of $R$ has at most two elements. 
\end{claim*}

To prove the claim above, note first that $R$ is clearly reflexive and symmetric and it is obviously compact. So, to see the remainder of the claim it will suffice 
to show that each $(y_n)\in \varprojlim_n ([k_n], f_n)$ is $R$-related to at most one element of $\varprojlim_n ([k_n], f_n)$ distinct 
from $(y_n)$. Towards a contradiction, assume otherwise, that is, assume that, with positive probability, there are $(x_n), (y_n), (z_n)\in \varprojlim_n ([k_n], f_n)$ such that 
\begin{equation*}
(x_n)\not=(y_n)\not= (z_n)\not= (x_n)\;\hbox{ and }\;(x_n)\, R\, (y_n)\, R\, (z_n). 
\end{equation*}
These relationships imply that, for large enough $n$, we have 
\begin{equation}\label{e:xyz}
x_n+1=y_n= z_n-1\;\hbox{ or }\; z_n+1=y_n=x_n-1. 
\end{equation}
Thus $1 < y_n < k_n$ for large enough $n$. Hence, using the definition of $f_n$, depending on the parity of $y_{n+1}$, either $f_n(y_{n+1}+1)$ or $f_n(y_{n+1}-1)$ is equal to $f_n(y_{n+1}) = y_n$ for large enough $n$. This means, using \eqref{e:xyz}, that either $f_n(x_{n+1}) = x_n$ or $f_n(z_{n+1}) = z_n$ is equal to $y_n$, contradicting \eqref{e:xyz} for large enough $n$. 

It follows from the above claim that, with probability $1$, $R$ is a compact equivalence relation whose equivalence classes have at most two elements. 
Thus, 
\begin{equation}\label{E:quot}
\bigl( \varprojlim_n\, ([k_n], f_n)\bigr)/R
\end{equation}
with the quotient topology is a compact metric space. 

\begin{claim*} With probability 1, $\varprojlim_n ([k_n], f_n)/R$ is a continuum. 
\end{claim*}

To see this claim, set $X= \varprojlim_n ([k_n], f_n)$. We need to show that, for any two non-empty, closed-and-open sets $U, V$ such that 
\[
U\cup V = X, 
\]
there exist sequences $(x_n)$ and $(y_n)$ with 
\begin{equation}\label{E:uv} 
(x_n) \in U,\, (y_n)\in V, \hbox{ and }\, (x_n)\, R\, (y_n). 
\end{equation}
Fix such $U$ and $V$. By compactness, there exists $n_0$ and sets $A, B$ such that 
\[
A\cup B = [k_{n_0}],\; U = \{ (x_n) \in X\mid x_{n_0}\in A\}\;\hbox{ and }\;V = \{ (y_n) \in X\mid y_{n_0}\in B\}. 
\]
Clearly, there are $x\in A$ and $y\in B$ with $|x-y|\leq 1$. Since each function in the sequence $(f_n)$ is a walk, one easily finds sequences 
$(x_n), (y_n)\in X$ such that $x_{n_0}=x$, $y_{n_0}=y$, and $|x_n-y_n|\leq 1$ for all $n$. It follows that \eqref{E:uv} holds for these sequences, and 
the claim is proved. 

The claim above allows us to see the space \eqref{E:quot} as a random continuum. 

\subsection*{Acknowledgement} We would like to thank Lionel Levine for pointing us in the direction of paper \cite{CK14}.


\begin{thebibliography}{9}
  \bibitem{B51} R.\,H.\;Bing, {\em Concerning hereditarily indecomposable continua}, Pacific J. Math. 1 (1951), 43--51.
 
  \bibitem{B} R.\,H.\;Bing, {\em The pseudo-arc}, in {\em Summary of Lectures and Seminars}, Summer Institute on Set Theoretic Topology, Madison 1955, revised 1958,  
  pp.\;72--74; in {\em The Collected Papers of R.H. Bing}, American Mathematical Society, 1988, pp.\;393--395.
  
  \bibitem{BKU00} L.\;Block, J.\;Keesling, V.\,V.\;Uspenskij, {\em Inverse limits which are the pseudoarc}, Houston J. Math. 26 (2000), 629--638.
  
  \bibitem{CK14} N.\;Curien, T.\;Konstantopoulos, {\em Iterating Brownian motions, ad libitum}, J. Theoret. Probab. 27 (2014), 433--448.
  
  \bibitem{E} R.\;Engelking, {\em General Topology}, Sigma Series in Pure Mathematics, 6, Heldermann Verlag, 1989.
  
  \bibitem{IS} T.\;Irwin, S.\;Solecki, {\em Projective Fra{\"i}ss{\'e} limits and the pseudo-arc}, Trans. Amer. Math. Soc. 358 (2006), 3077--3096.
    
   \bibitem{Ke} A.\,S.\;Kechris, {\em Classical Descriptive Set Theory},  Graduate Texts in Mathematics, 156, Springer Verlag, 1995. 
   
  \bibitem{MP10} P.\;M{\"o}rters, Y.\;Peres, \emph{Brownian Motion}, Cambridge Series in Statistical and Probabilistic Mathematics, 30, Cambridge University Press, 2010.
  
  \bibitem{N} S.\,B.\;Nadler, {\em Continuum Theory. An Introduction}, Monographs and Textbooks in Pure and Applied Mathematics, 158, Marcel Dekker, 1992.
  
  \bibitem{P} J.\;Prajs, {\em Open problems in the study of homogeneous continua}, plenary talk at the 52nd Spring Topology and Dynamical Systems Conference, 
  Auburn University, March 17, 2018. 
\end{thebibliography}
\end{document}